\documentclass[11pt,a4paper,reqno]{amsart} 

\usepackage{hyperref}
\usepackage[capitalize, nameinlink, noabbrev]{cleveref}
\usepackage{autonum}
\usepackage{tikz}
\usepackage{mathtools,bm,bbm,amsthm, amssymb, amsmath, amsfonts}
\usepackage{dsfont} 
\usepackage{color}
\usepackage{animate}
\usepackage{etoolbox}
\usepackage{ulem}
\usepackage{comment}
\usepackage{pgf}
\usetikzlibrary{calc}
\usetikzlibrary{patterns}
\usetikzlibrary{arrows}
\usetikzlibrary{decorations.pathreplacing}
\usepackage{fancyvrb}
\usepackage[utf8]{inputenc}
\usepackage{pgfplots}

\usepackage{microtype}
\allowdisplaybreaks

\usepackage[T1]{fontenc}
\usepackage{lmodern}
\usepackage{mathtools,bm,bbm,amsthm, amssymb, amsmath, amsfonts}

\usepackage{fullpage}

\definecolor{wiasblue}   {cmyk}{1.0, 0.60, 0, 0}

\def\mc{\mathcal}
\def\ms{\mathsf}
\def\mb{\mathbb}

\def\Z{\mb Z}
\def\E{\mb E}
\def\P{\mb P}

\def\R{\mb R}

\def\a{\alpha}
\def\de{\delta}
\def\De{\Delta}
\def\e{\varepsilon}

\def\r{\varrho}
\def\ra{\rightarrow}
\def\longra{\longrightarrow}
\def\es{\emptyset}
\def\mc{\mathcal}

\def\Fe{\mc F_{e^c}}
\def\C{\mc C}
\def\pol{\ms{pol}}
\def\been{\begin{enumerate}}
\def\enen{\end{enumerate}}

\def\deg{\ms{deg}}

\ifx\@NOAMS\undefined%
 \newtheoremstyle{ejpecpbodyit}
 {3pt}
 {3pt}
 {\itshape}
 {}
 {\bfseries\sffamily}
 {.}
 { }
 {}
 \newtheoremstyle{ejpecpbodyrm}
 {3pt}
 {3pt}
 {}
 {}
 {\bfseries\sffamily}
 {.}
 { }
 {}
\fi
\ifx\@NOAMS\undefined\theoremstyle{ejpecpbodyit}\fi%
\newtheorem{theorem}{Theorem}[section]%
\newtheorem{conjecture}[theorem]{Conjecture}%
\newtheorem{lemma}[theorem]{Lemma}%
\newtheorem{proposition}[theorem]{Proposition}%
\ifx\@NOAMS\undefined\theoremstyle{ejpecpbodyrm}\fi%
\newtheorem{example}[theorem]{Example}%

\date{\today}

\begin{document}

\title{Weakly reinforced P\'olya urns on countable networks}
\author{Yannick Couzini\'e}
\author{Christian Hirsch}
\address[Yannick Couzini\'e]{Dipartimento di Matematica e Fisica,
                             Università Roma Tre, Largo S.L.~Murialdo
                             00146 Rome, Italy.}
	                 \email{yannick.couzinie@uniroma3.it}
\address[Christian Hirsch]{University of Groningen, Bernoulli Institute, Nijenborgh 9, 9747 AG Groningen, The Netherlands.}
\email{c.p.hirsch@rug.nl}

\begin{abstract}We study the long-time asymptotics of a network of weakly reinforced P\'olya urns. In this system, which extends the WARM introduced by R.~van der Hofstad et.~al.~(2016) to countable networks, the nodes fire at times given by a Poisson point process. When a node fires, one of the incident edges is selected with a probability proportional to its weight raised to a power $\a < 1$, and then this weight is increased by 1. 

We show that for $\a < 1/2$ on a network of bounded degrees, every edge is reinforced a positive proportion of time, and that the limiting proportion can be interpreted as an equilibrium in a countable network. Moreover, in the special case of regular graphs, this homogenization remains valid beyond the threshold $\a = 1/2$.
\end{abstract}

\maketitle
\section{Introduction}\label{sec:intro}
%
%
P\'olya's urn process is the paradigm model for a random process incorporating reinforcement effects. However, when thinking in the direction of applications, a single urn often does not represent complex interactions accurately. For instance, in the field of social sciences, the formation of friendship networks could be related to reinforcement effects in social interactions \cite{rolles}. In economy, in a network of companies competing on a variety of products, the global reputation could influence the market shares of the products differently \cite{benaim, lucas}. Finally, in the domain of neuroscience, it is plausible that synapses that were successful in the past should be selected again in the future and reinforced with a higher probability \cite{warm1}.

%
%
In the present paper, we focus on the network formation model proposed in \cite{warm1} and study its long-time behavior in the regime of weak reinforcement. In that model, starting from a base network, nodes are picked sequentially at random. Once a node is selected, we choose one of the incident edges with a probability proportional to its weight to some power $\a > 0$ and increase that weight by 1. The analysis of \cite{warm1} concerns the asymptotic proportion of times that edges are reinforced in the regime of strong reinforcement, where $\a > 1$. In this regime, the limiting proportion is random and coincides with some stable equilibrium in an associated dynamical system, which is typically concentrated on a small subset of the edges if $\a$ is large.

In contrast, we consider the regime of weak reinforcement, where $\a < 1$ and find that in many settings the reinforcement proportions converge to a uniquely determined limit equilibrium, which is supported on all edges of the base network. Figure \ref{frontFig} illustrates the network at an early and at a late time point.

\begin{figure}[!htpb]
	\begin{tikzpicture}[scale = .7]
\fill (0, 0) circle (3pt);
\fill (0, 1) circle (3pt);
\fill (0, 2) circle (3pt);
\fill (0, 3) circle (3pt);
\fill (0, 4) circle (3pt);
\fill (0, 5) circle (3pt);
\fill (0, 6) circle (3pt);
\fill (0, 7) circle (3pt);
\fill (0, 8) circle (3pt);
\fill (0, 9) circle (3pt);
\fill (1, 0) circle (3pt);
\fill (1, 1) circle (3pt);
\fill (1, 2) circle (3pt);
\fill (1, 3) circle (3pt);
\fill (1, 4) circle (3pt);
\fill (1, 5) circle (3pt);
\fill (1, 6) circle (3pt);
\fill (1, 7) circle (3pt);
\fill (1, 8) circle (3pt);
\fill (1, 9) circle (3pt);
\fill (2, 0) circle (3pt);
\fill (2, 1) circle (3pt);
\fill (2, 2) circle (3pt);
\fill (2, 3) circle (3pt);
\fill (2, 4) circle (3pt);
\fill (2, 5) circle (3pt);
\fill (2, 6) circle (3pt);
\fill (2, 7) circle (3pt);
\fill (2, 8) circle (3pt);
\fill (2, 9) circle (3pt);
\fill (3, 0) circle (3pt);
\fill (3, 1) circle (3pt);
\fill (3, 2) circle (3pt);
\fill (3, 3) circle (3pt);
\fill (3, 4) circle (3pt);
\fill (3, 5) circle (3pt);
\fill (3, 6) circle (3pt);
\fill (3, 7) circle (3pt);
\fill (3, 8) circle (3pt);
\fill (3, 9) circle (3pt);
\fill (4, 0) circle (3pt);
\fill (4, 1) circle (3pt);
\fill (4, 2) circle (3pt);
\fill (4, 3) circle (3pt);
\fill (4, 4) circle (3pt);
\fill (4, 5) circle (3pt);
\fill (4, 6) circle (3pt);
\fill (4, 7) circle (3pt);
\fill (4, 8) circle (3pt);
\fill (4, 9) circle (3pt);
\fill (5, 0) circle (3pt);
\fill (5, 1) circle (3pt);
\fill (5, 2) circle (3pt);
\fill (5, 3) circle (3pt);
\fill (5, 4) circle (3pt);
\fill (5, 5) circle (3pt);
\fill (5, 6) circle (3pt);
\fill (5, 7) circle (3pt);
\fill (5, 8) circle (3pt);
\fill (5, 9) circle (3pt);
\fill (6, 0) circle (3pt);
\fill (6, 1) circle (3pt);
\fill (6, 2) circle (3pt);
\fill (6, 3) circle (3pt);
\fill (6, 4) circle (3pt);
\fill (6, 5) circle (3pt);
\fill (6, 6) circle (3pt);
\fill (6, 7) circle (3pt);
\fill (6, 8) circle (3pt);
\fill (6, 9) circle (3pt);
\fill (7, 0) circle (3pt);
\fill (7, 1) circle (3pt);
\fill (7, 2) circle (3pt);
\fill (7, 3) circle (3pt);
\fill (7, 4) circle (3pt);
\fill (7, 5) circle (3pt);
\fill (7, 6) circle (3pt);
\fill (7, 7) circle (3pt);
\fill (7, 8) circle (3pt);
\fill (7, 9) circle (3pt);
\fill (8, 0) circle (3pt);
\fill (8, 1) circle (3pt);
\fill (8, 2) circle (3pt);
\fill (8, 3) circle (3pt);
\fill (8, 4) circle (3pt);
\fill (8, 5) circle (3pt);
\fill (8, 6) circle (3pt);
\fill (8, 7) circle (3pt);
\fill (8, 8) circle (3pt);
\fill (8, 9) circle (3pt);
\fill (9, 0) circle (3pt);
\fill (9, 1) circle (3pt);
\fill (9, 2) circle (3pt);
\fill (9, 3) circle (3pt);
\fill (9, 4) circle (3pt);
\fill (9, 5) circle (3pt);
\fill (9, 6) circle (3pt);
\fill (9, 7) circle (3pt);
\fill (9, 8) circle (3pt);
\fill (9, 9) circle (3pt);
\draw[dashed] (0, 0) -- (0, 1);
\draw[dashed] (0, 1) -- (0, 2);
\draw[dashed] (0, 2) -- (0, 3);
\draw[line width = 1.0mm] (0, 3) -- (0, 4);
\draw[line width = 1.0mm] (0, 4) -- (0, 5);
\draw[line width = 1.0mm] (0, 5) -- (0, 6);
\draw[line width = 1.0mm] (0, 6) -- (0, 7);
\draw[line width = 1.0mm] (0, 7) -- (0, 8);
\draw[dashed] (0, 8) -- (0, 9);
\draw[dashed] (1, 0) -- (1, 1);
\draw[dashed] (1, 1) -- (1, 2);
\draw[line width = 2.0mm] (1, 2) -- (1, 3);
\draw[dashed] (1, 3) -- (1, 4);
\draw[dashed] (1, 4) -- (1, 5);
\draw[dashed] (1, 5) -- (1, 6);
\draw[dashed] (1, 6) -- (1, 7);
\draw[dashed] (1, 7) -- (1, 8);
\draw[dashed] (1, 8) -- (1, 9);
\draw[line width = 1.0mm] (2, 0) -- (2, 1);
\draw[dashed] (2, 1) -- (2, 2);
\draw[line width = 2.0mm] (2, 2) -- (2, 3);
\draw[line width = 1.5mm] (2, 3) -- (2, 4);
\draw[dashed] (2, 4) -- (2, 5);
\draw[line width = 1.0mm] (2, 5) -- (2, 6);
\draw[dashed] (2, 6) -- (2, 7);
\draw[line width = 1.0mm] (2, 7) -- (2, 8);
\draw[dashed] (2, 8) -- (2, 9);
\draw[line width = 1.0mm] (3, 0) -- (3, 1);
\draw[dashed] (3, 1) -- (3, 2);
\draw[dashed] (3, 2) -- (3, 3);
\draw[line width = 1.0mm] (3, 3) -- (3, 4);
\draw[dashed] (3, 4) -- (3, 5);
\draw[dashed] (3, 5) -- (3, 6);
\draw[dashed] (3, 6) -- (3, 7);
\draw[dashed] (3, 7) -- (3, 8);
\draw[line width = 1.5mm] (3, 8) -- (3, 9);
\draw[dashed] (4, 0) -- (4, 1);
\draw[dashed] (4, 1) -- (4, 2);
\draw[dashed] (4, 2) -- (4, 3);
\draw[dashed] (4, 3) -- (4, 4);
\draw[line width = 1.0mm] (4, 4) -- (4, 5);
\draw[line width = 1.0mm] (4, 5) -- (4, 6);
\draw[line width = 2.0mm] (4, 6) -- (4, 7);
\draw[line width = 1.0mm] (4, 7) -- (4, 8);
\draw[dashed] (4, 8) -- (4, 9);
\draw[dashed] (5, 0) -- (5, 1);
\draw[dashed] (5, 1) -- (5, 2);
\draw[dashed] (5, 2) -- (5, 3);
\draw[dashed] (5, 3) -- (5, 4);
\draw[dashed] (5, 4) -- (5, 5);
\draw[line width = 1.0mm] (5, 5) -- (5, 6);
\draw[dashed] (5, 6) -- (5, 7);
\draw[dashed] (5, 7) -- (5, 8);
\draw[dashed] (5, 8) -- (5, 9);
\draw[line width = 1.5mm] (6, 0) -- (6, 1);
\draw[dashed] (6, 1) -- (6, 2);
\draw[dashed] (6, 2) -- (6, 3);
\draw[line width = 1.0mm] (6, 3) -- (6, 4);
\draw[dashed] (6, 4) -- (6, 5);
\draw[dashed] (6, 5) -- (6, 6);
\draw[dashed] (6, 6) -- (6, 7);
\draw[line width = 1.0mm] (6, 7) -- (6, 8);
\draw[dashed] (6, 8) -- (6, 9);
\draw[dashed] (7, 0) -- (7, 1);
\draw[dashed] (7, 1) -- (7, 2);
\draw[line width = 1.5mm] (7, 2) -- (7, 3);
\draw[line width = 1.0mm] (7, 3) -- (7, 4);
\draw[dashed] (7, 4) -- (7, 5);
\draw[dashed] (7, 5) -- (7, 6);
\draw[line width = 1.0mm] (7, 6) -- (7, 7);
\draw[line width = 1.0mm] (7, 7) -- (7, 8);
\draw[dashed] (7, 8) -- (7, 9);
\draw[line width = 2.5mm] (8, 0) -- (8, 1);
\draw[dashed] (8, 1) -- (8, 2);
\draw[dashed] (8, 2) -- (8, 3);
\draw[dashed] (8, 3) -- (8, 4);
\draw[dashed] (8, 4) -- (8, 5);
\draw[line width = 1.0mm] (8, 5) -- (8, 6);
\draw[dashed] (8, 6) -- (8, 7);
\draw[dashed] (8, 7) -- (8, 8);
\draw[line width = 1.0mm] (8, 8) -- (8, 9);
\draw[dashed] (9, 0) -- (9, 1);
\draw[line width = 1.0mm] (9, 1) -- (9, 2);
\draw[dashed] (9, 2) -- (9, 3);
\draw[line width = 1.0mm] (9, 3) -- (9, 4);
\draw[line width = 1.5mm] (9, 4) -- (9, 5);
\draw[dashed] (9, 5) -- (9, 6);
\draw[dashed] (9, 6) -- (9, 7);
\draw[line width = 1.0mm] (9, 7) -- (9, 8);
\draw[dashed] (9, 8) -- (9, 9);
\draw[dashed] (0, 0) -- (1, 0);
\draw[dashed] (1, 0) -- (2, 0);
\draw[dashed] (2, 0) -- (3, 0);
\draw[dashed] (3, 0) -- (4, 0);
\draw[dashed] (4, 0) -- (5, 0);
\draw[line width = 1.0mm] (5, 0) -- (6, 0);
\draw[dashed] (6, 0) -- (7, 0);
\draw[dashed] (7, 0) -- (8, 0);
\draw[dashed] (8, 0) -- (9, 0);
\draw[dashed] (0, 1) -- (1, 1);
\draw[dashed] (1, 1) -- (2, 1);
\draw[line width = 1.0mm] (2, 1) -- (3, 1);
\draw[line width = 1.0mm] (3, 1) -- (4, 1);
\draw[dashed] (4, 1) -- (5, 1);
\draw[dashed] (5, 1) -- (6, 1);
\draw[line width = 1.0mm] (6, 1) -- (7, 1);
\draw[line width = 1.0mm] (7, 1) -- (8, 1);
\draw[dashed] (8, 1) -- (9, 1);
\draw[line width = 1.0mm] (0, 2) -- (1, 2);
\draw[line width = 1.0mm] (1, 2) -- (2, 2);
\draw[line width = 1.5mm] (2, 2) -- (3, 2);
\draw[dashed] (3, 2) -- (4, 2);
\draw[dashed] (4, 2) -- (5, 2);
\draw[dashed] (5, 2) -- (6, 2);
\draw[dashed] (6, 2) -- (7, 2);
\draw[line width = 1.0mm] (7, 2) -- (8, 2);
\draw[line width = 1.0mm] (8, 2) -- (9, 2);
\draw[dashed] (0, 3) -- (1, 3);
\draw[line width = 1.5mm] (1, 3) -- (2, 3);
\draw[dashed] (2, 3) -- (3, 3);
\draw[dashed] (3, 3) -- (4, 3);
\draw[line width = 2.0mm] (4, 3) -- (5, 3);
\draw[line width = 1.0mm] (5, 3) -- (6, 3);
\draw[dashed] (6, 3) -- (7, 3);
\draw[line width = 1.0mm] (7, 3) -- (8, 3);
\draw[line width = 2.0mm] (8, 3) -- (9, 3);
\draw[dashed] (0, 4) -- (1, 4);
\draw[dashed] (1, 4) -- (2, 4);
\draw[dashed] (2, 4) -- (3, 4);
\draw[dashed] (3, 4) -- (4, 4);
\draw[dashed] (4, 4) -- (5, 4);
\draw[line width = 1.0mm] (5, 4) -- (6, 4);
\draw[dashed] (6, 4) -- (7, 4);
\draw[line width = 2.0mm] (7, 4) -- (8, 4);
\draw[dashed] (8, 4) -- (9, 4);
\draw[dashed] (0, 5) -- (1, 5);
\draw[dashed] (1, 5) -- (2, 5);
\draw[line width = 1.0mm] (2, 5) -- (3, 5);
\draw[dashed] (3, 5) -- (4, 5);
\draw[line width = 1.5mm] (4, 5) -- (5, 5);
\draw[dashed] (5, 5) -- (6, 5);
\draw[line width = 1.5mm] (6, 5) -- (7, 5);
\draw[dashed] (7, 5) -- (8, 5);
\draw[line width = 1.0mm] (8, 5) -- (9, 5);
\draw[line width = 1.0mm] (0, 6) -- (1, 6);
\draw[dashed] (1, 6) -- (2, 6);
\draw[dashed] (2, 6) -- (3, 6);
\draw[line width = 1.0mm] (3, 6) -- (4, 6);
\draw[dashed] (4, 6) -- (5, 6);
\draw[line width = 1.0mm] (5, 6) -- (6, 6);
\draw[dashed] (6, 6) -- (7, 6);
\draw[line width = 1.0mm] (7, 6) -- (8, 6);
\draw[dashed] (8, 6) -- (9, 6);
\draw[line width = 1.0mm] (0, 7) -- (1, 7);
\draw[dashed] (1, 7) -- (2, 7);
\draw[line width = 1.0mm] (2, 7) -- (3, 7);
\draw[line width = 1.0mm] (3, 7) -- (4, 7);
\draw[dashed] (4, 7) -- (5, 7);
\draw[line width = 1.5mm] (5, 7) -- (6, 7);
\draw[dashed] (6, 7) -- (7, 7);
\draw[dashed] (7, 7) -- (8, 7);
\draw[line width = 1.0mm] (8, 7) -- (9, 7);
\draw[line width = 1.5mm] (0, 8) -- (1, 8);
\draw[dashed] (1, 8) -- (2, 8);
\draw[line width = 1.0mm] (2, 8) -- (3, 8);
\draw[dashed] (3, 8) -- (4, 8);
\draw[dashed] (4, 8) -- (5, 8);
\draw[dashed] (5, 8) -- (6, 8);
\draw[dashed] (6, 8) -- (7, 8);
\draw[dashed] (7, 8) -- (8, 8);
\draw[dashed] (8, 8) -- (9, 8);
\draw[line width = 1.0mm] (0, 9) -- (1, 9);
\draw[dashed] (1, 9) -- (2, 9);
\draw[dashed] (2, 9) -- (3, 9);
\draw[dashed] (3, 9) -- (4, 9);
\draw[line width = 1.0mm] (4, 9) -- (5, 9);
\draw[dashed] (5, 9) -- (6, 9);
\draw[dashed] (6, 9) -- (7, 9);
\draw[line width = 1.0mm] (7, 9) -- (8, 9);
\draw[dashed] (8, 9) -- (9, 9);
\end{tikzpicture}\hspace{0.8cm}
	\input{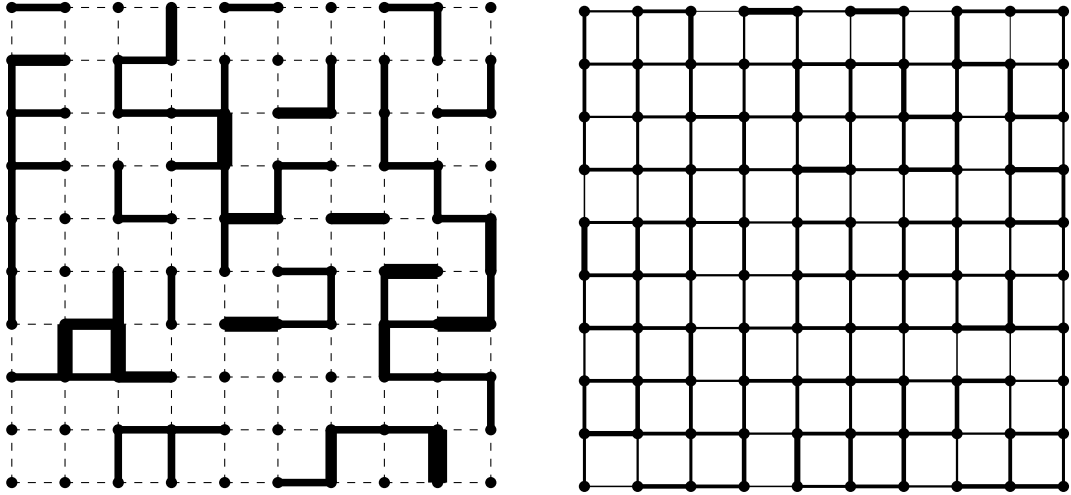}
	\caption{Weakly reinforced P\'olya model after time $2$ (left) and time $100$ (right). Dashed edges have not been reinforced until the considered time. The widths of solid edges are proportional to time-normalized weights.}
	\label{frontFig}
\end{figure}

Hence, together with the analysis in \cite{warm1, warm2}, our main result is a first step towards a network-based analog of Rubin's dichotomy for classical P\'olya urns:\, while for strong reinforcement some edges are only reinforced finitely often, in the weakly reinforced regime all edges are reinforced a positive proportion of time. Although this description outlines the broader picture, more research is needed to carve out the precise conditions for this dichotomy. Indeed, \cite{rubin} describes an example of a meticulously designed network together with firing rates where even for $\a > 1$ there is percolation of edges that are reinforced a positive proportion of time.

%
%
Similarly to the setting of \cite{main_paper}, one major difficulty in the analysis stems from the choice of the base graph. In contrast to \cite{benaim, warm1}, we do not restrict our attention to finite base graphs, but work on a countable network with uniformly bounded degrees. In particular, the highly developed machinery of stochastic approximation algorithms invoked in \cite{benaim, warm1} is not available as the state space of the associated dynamical system would be infinite-dimensional. In \cite{main_paper}, the very strong reinforcement leads to an effective decomposition of the countable network into finite islands separated by edges that are never reinforced. This trick recovers the finite-dimensional setting. However, the strategy cannot work in the regime of weak reinforcement, where we expect all edges to be reinforced a positive proportion of time.

%
%
Hence, in the present paper, we follow an entirely different plan to prove our
main result. Namely, convergence of the normalized edge weights to an
equilibrium for (1) all graphs of uniformly bounded degree if $\a < 1/2$, (2)
regular graphs if $\a < 1/2+\varepsilon$ for some $\varepsilon>0$, and (3) $\Z$ if $\a < 1$. This plan consists of three critical steps. First, we invoke a compactness argument to obtain an equilibrium on the entire countable network. Second, we establish that all edges are reinforced a positive proportion of time. This step rests on a percolation argument, where we decompose the network again into finite islands separated by edges that are now reinforced a positive proportion of time. Finally, to obtain convergence to the equilibrium, we rely on a homogenizing bootstrap argument. It formalizes the intuition that for weak reinforcement, deviations from the equilibrium decrease over time.

%
%
The rest of the manuscript is organized as follows. \cref{sec:results} introduces the model and states
the main contribution of the paper, a homogenization result in the sub-linear
regime for three different combinations of graphs and exponents. In
\cref{sec:bounded}, we consider graphs of bounded degree and $\a < 1/2$.
Finally, \cref{seg:regular} deals with regular graphs of degree $\Delta$ and covers $\a <
1/2+\varepsilon$ for $\De  >  2$ and $\a < 1$ for the graph $\Z$.

\section{Model definition and main result}\label{sec:results}
Let $G  =  (V, E)$ be a countable graph with uniformly bounded degrees. That is, the vertex set $V$ is countable and 
$$\De := \sup_{v \in V} \deg(v) < \infty.$$
If $\De=\deg(v)$ for every $v\in V$ then the graph is regular with degree
$\Delta$, which we write as $\Delta$-regular.

We investigate a system of random variables
\begin{align}
	N_t := {\{N_t(e)\}}_{e\in E}
\end{align}
of interacting P\'{o}lya-type urns on the edge set $E$ at continuous time $t \ge 0$ on a probability space $(\Omega, \P)$. The
dynamics of $N_t$ are a continuous-time analog of the process considered
in~\cite{warm1}. Loosely speaking, every vertex has a Poisson clock and
whenever that clock rings at a node $v \in V$, the dynamics choose and increment the weights on one
of the incident edges $E_v = \{e\in E:\,v\in e\}$ by 1.
The choice of the edge happens using the power-weighted P\'{o}lya-scheme
\begin{align}\label{eqn:polya_weight}
        \pol_{v,e}(N_t): = \frac{{N_t(e)}^\a}
                                      {\sum_{e'\in E_v}N_t(e')^\a}\;,
                                      \qquad \a\ge 0\;.
\end{align}
Henceforth, we tacitly assume that $\a < 1$.
More precisely, the weights $N_t(e)$ are initialized at $1$ and the
dynamics of $N_t$ are governed by an iid family of Poisson
processes ${\{P_v\}}_{v\in V}$ on $[0, \infty)\times[0,1]$ with
intensity $1$ counting clock-ring events at vertex $v\in V$
in a time window $[0,\infty)$ with marks in the range
$[0,1]$.
If the process $P_v$ contains an atom of the form $(t, u) \in [0, \infty) \times [0, 1]$, then increment the mass of an edge $e_i \in E_v = \{e_1, \dots e_{\deg(v)}\}$ by 1 if $u\in U_{v,e_i}$, where ${\{U_{v,
e_i}\}}_{i \le \deg(v)}$ is a partition of $[0,1]$ given by
\begin{align}
	U_{v, e_i} & =  \Big(\sum_{j \le i - 1}\pol_{v, e_{j}}(N_t),\sum_{j \le i }
        \pol_{v, e_{j}}(N_t) \Big]\;.
\end{align}
For the existence of the process $\{N_t\}_{t \ge 0}$ on bounded-degree graphs, see \cite{main_paper}.

Finally, a non-negative vector $\mu\in\R_+^E$ defines an \emph{equilibrium on $G$} if
\begin{align}\label{eqn:equi}
        \mu(e)  :=  \sum_{v\in e}\frac{{\mu(e)}^\a}
                                    {\sum_{e'\in E_v}{\mu(e')}^\a}
\end{align}
holds for all $e\in E$ with $\mu(e) > 0$. This is a straightforward extension
of the notion of finite equilibria from~\cite{warm1} to countable networks. Now, let
\begin{align}
	X_t:=\frac{N_t}t
\end{align}
denote the time-normalized system of weights at time $t > 0$. We say that $X_t$ exhibits
\emph{homogenization} if the equilibrium measure $\mu$ exists, is unique
and $X_t \to \mu$ almost surely. 

\begin{conjecture}[Homogenization]
	Let $G$ be a countable graph with uniformly bounded degrees. Then, $X_t$ exhibits homogenization.
\end{conjecture}

In the present work, we verify this conjecture for three combinations of $G$ and $\a$.

\begin{theorem}[Homogenization]\label{mainThm}
$X_t$ exhibits homogenization in the following cases.
\begin{enumerate}
\item The degree of $G$ is uniformly bounded and $\a < 1/2$.
\item $G$ is $\De$-regular and $\a < 1/2+\varepsilon$ for some $\varepsilon>0$.
\item $G = \Z$ and $\a < 1$.
\end{enumerate}
\end{theorem}

\section{Proof of \cref{mainThm} (1)}\label{sec:bounded}
In this section, we establish part (1) of \cref{mainThm}. That is, we show
homogenization for $\a < 1/2$ and graphs of bounded degree. First, in
\cref{sec:existence_equi}, we prove existence of a non-vanishing equilibrium.
Then, in \cref{sec:bounded_limit}, we show that $X_t$ converges to such an
equilibrium, which is in fact uniquely determined.
%
%
\subsection{Existence of equilibrium}\label{sec:existence_equi}

%
%
Henceforth, we call an equilibrium $\mu \in \R_+^E$ \emph{non-vanishing} if $\mu(e) > 0$ for every $e \in E$.
Before discussing existence of non-vanishing equilibria on general graphs, we present the
$\De$-regular case as a particularly easy example.
\begin{example}[$\De$-regular graphs]\label{lemma:equi_zd}
 Let $G$ be a $\De$-regular graph. Then, $\mu \equiv 2/\De$ defines a non-vanishing equilibrium.
Indeed, checking \cref{eqn:equi} leads to
 $$
  \sum_{v\in e}
 \frac{\mu^\a}{\sum_{e'\in E_v}\mu^\a}\\
  = \sum_{v\in e}
 \frac1{\sum_{e'\in E_v}1}\\
  = \frac2\De = \mu.
 $$
\end{example}

%
%
 The general case of bounded-degree graphs does not admit such an easy analysis as
the equilibrium could in principle assume an infinite number of different values. Nevertheless, we deduce existence from a compactness argument.
\begin{proposition}[Existence of non-vanishing equilibria]\label{lemma:existence_for_all}
	Every graph with degrees uniformly bounded by $\Delta$ exhibits at least one non-vanishing
    equilibrium $\mu$ {{with $\mu(e) \ge 2/\De^{1/(1-\a)}$ for all
    $e\in E$}}.
\end{proposition}
\begin{proof}
	%
	%
	By \cref{eqn:equi}, the equilibria correspond to fixed points of the continuous operator
 \begin{align}
 T\colon \R_+^E&\longra \R_+^E\\
 \mu(\cdot)&\longmapsto
 \sum_{v\in \cdot}
 \frac{{\mu(\cdot)}^\a}
  {\sum_{e'\in E_v}{\mu(e')}^\a}.
 \end{align}

	%
	%
	As a product of locally convex and Hausdorff spaces, $\R_+^E$ is again locally convex and Hausdorff. 
Now, define the closed set 
    {$C = {[2\De^{-1/(1 - \a)}, 2]}^E$},
	which as a product of convex and
    compact sets, is again convex and compact. We claim that 
		$T(C) \subseteq C.$
	Once this claim is established, Schauder's fixed point theorem yields a
    non-vanishing equilibrium in $C$.

	To prove $T(C) \subseteq C$, note that for every $\mu\in C$ and $e = \{v_1, v_2\} \in E$ we have
 \begin{align}
	 T(\mu)(e)
  = 
 \frac{\mu(e)^\a}
	 {\sum_{e'\in E_{v_1}}{\mu(e')}^\a} + \frac{\mu(e)^\a}
	 {\sum_{e'\in E_{v_2}}{\mu(e')}^\a}
  \le 2\;, 
 \end{align}
 and
 \begin{align}
	 T(\mu)(e)
  = \sum_{i \le 2}
 \frac{\mu(e)^\a}
	 {\sum_{e'\in E_{v_i}}{\mu(e')}^\a}
 \ge \sum_{i \le 2}
 \frac{\mu(e)^\a}
  {2^\a\deg(v_i)}
  \ge \frac{2\mu(e)^\a}{2^\a {\De}}
  \ge \frac2{{{\De}}^{1/(1 - \a)}},
 \end{align}
 as claimed.
\end{proof}

%
%
\subsection{Proof of Theorem \ref{mainThm} (1)}\label{sec:bounded_limit}
The proof of part (1) of Theorem \ref{mainThm} is based on three pivotal ingredients. 

%
%
First, as each edge is incident to two vertices firing at rate 1, its weight grows at rate at most 2. In other words, for every $e \in E$, almost surely,
$$X_+(e) : = \limsup_{t \to \infty}X_t(e) \le 2.$$
\begin{lemma}[Rate upper bound]\label{lemma:loc_firstboundonC}
	Let $e \in E$ be arbitrary. Then, $\P(X_+(e) \le 2) = 1$.
\end{lemma}

%
%
Second, we derive a positive lower bound for the growth rate
\begin{align}
        X_-(e) :=\liminf_{t\rightarrow\infty}X_t(e).
\end{align}
 This task is slightly more subtle than the upper bound, since arbitrarily
 small rates can occur with positive probability. To approach this challenge,
 we fix throughout this section a non-vanishing equilibrium $\mu$  as in
 \cref{lemma:existence_for_all} and write 
$$E_e := \{e' \in E:\, e' \cap e \ne \es\}$$
for the family of edges adjacent to a given edge $e \in E$. Then, an edge $e \in E$ is \emph{$\de$-stable} if
\begin{align}
	\min_{e' \in E_e} \frac{X_-(e')}{\mu(e')} \ge \de.
\end{align}

\begin{lemma}[Rate lower bound]
	\label{noPercLem}
	There exists $\de_0 > 0$ such that with probability $1$,
    all connected components of $\de_0$-unstable edges are finite.
\end{lemma}

%
%
Finally, we rely on a bootstrapping procedure to push the weights iteratively closer to 1. 
To that end, define 
\begin{align}
	\C(e) : = [X_-(e), X_+(e)]
\end{align}
as the smallest interval containing the accumulation points of $X_t(e)$.

\begin{lemma}[Bootstrap on bounded degree graphs]
	\label{bootStrapLem}
    Let $\a < 1/2$ and $\r>1$. If $e\in E$ is such that $\C(e') \subseteq
    [\r^{-1}\mu(e'), \r\mu(e')]$ holds for all $e' \in E_e$, then $\C(e)
    \subseteq [\r^{-2\a}\mu(e), \r^{2\a}\mu(e)]$.
\end{lemma}

Before establishing Lemmas~\ref{lemma:loc_firstboundonC}--\ref{bootStrapLem},
we elucidate how they can be combined to yield the proof of part (1) of \cref{mainThm}.

%
%
\begin{proof}[Proof of \cref{mainThm}(1)]
        Let $\de_0 > 0$ be as in \cref{noPercLem}. { Assume without
                loss of generality that $\de_0 \le \De^{-1/(1-\a)}$ so that $2\le
        \mu(e)\De^{1/(1-\a)}\le \mu(e)/\de_0$ for any $e\in E$ where the first
        inequality follows from $\mu(e)\ge 2/\Delta^{1/(1-\a)}$ given by
        \cref{lemma:existence_for_all}}. We first claim that $X_-(e) \ge \de_0\mu(e)$
        almost surely for every $e \in E$.  { Indeed, assume the contrary
        and let $S$ be a connected component of $\de_0$-unstable edges. By
        \cref{noPercLem}, $S$ is finite and we choose $e^*\in S$ (randomly)
        such that $\r:=\frac{\mu(e^*)}{X_-(e^*)}$ is maximal.} Then, $\r^{-1}
        \ge \de_0$ almost surely since otherwise \cref{lemma:loc_firstboundonC}
        and the bootstrap property of \cref{bootStrapLem}{,
        combined with $2\le \mu(e)/\delta_0\le \r\mu(e)$,} would give that
        $\C(e^*) \subseteq [\r^{-2\a}\mu(e^*), 2]$, thereby contradicting that
        \[\frac{X_-(e^*) }{ \mu(e^*)} =
        \r^{-1} < \r^{-2\a}.\]
    {{Hence, the infimum
    \[\r_* := \inf\{\r > 1\colon\C(e) \subseteq [\r^{-1}\mu(e),
    \r\mu(e)] \text{ for every $e \in E$}\}\]
	is at most $\de_0^{-1}$. }}Now,   $\r_* = 1$ almost surely
    since otherwise the bootstrap property yields that $\C(e)
	\subseteq [\r_*^{-2\a}\mu(e), \r_*^{2\a}\mu(e)]$ for every $e \in
    E$, thereby contradicting the choice of $\r_*$.
\end{proof}
%
%
It remains to establish the auxiliary results. We start by proving Lemma \ref{lemma:loc_firstboundonC}.
\begin{proof}[Proof of Lemma \ref{lemma:loc_firstboundonC}]
One extremal case is that each clock-ring event for $P_v$ increments
the value for an edge $e = \{v, w\}\in E$.
The process
\begin{align}
	Y_t := \frac{P^{}_v([0, t]\times[0, 1])
 + P_w([0, t]\times[0, 1])}t
\end{align}
counts the normalized occurrences for this upper bound and has an expected value of
\begin{align}
 \E[Y_t]
 = \frac{\E[P^{}_v([0, t]\times [0, 1])]
  + \E[P^{}_w([0, t]\times [0, 1])]}t
 = \frac{2t}t = 2\;.
\end{align}
Then, the strong law of large numbers for homogeneous Poisson point
processes, gives that almost surely
$ \lim_{t \ra \infty}Y_t = 2$.
Now, 
$
 X_t(e)\le Y_t + \frac1t
 $
for all $t \ge 0$, implies that almost surely,
$
 \limsup_{t \ra \infty} X_t(e)\le \lim_{t \ra \infty}
 Y_t = 2.
 $
\end{proof}

%
%
Next, we verify the bootstrap property.
\begin{proof}[Proof of Lemma \ref{bootStrapLem}]
    Let $\r > 1$ and $e = \{v_1, v_2\} \in E$ be such that $\C(e') \subseteq
    [\r^{-1}\mu(e'), \r\mu(e')]$ for all $e' \in E_e$. Moreover, for $\e > 0$
    set $\r_\e = (1 + \e) \r$. Then, there exists a random time $T  < \infty $
    such that $X_t(e') \in [\r_\e^{-1}\mu(e'), \r_\e\mu(e')]$ for all $t \ge T$
    and $e' \in E_e$. In particular, for every $i \in \{1, 2\}$,
    {
    \begin{align}
            \pol_{v_i,e}(X_t)
            = \frac{{N_t(e)}^\a}
                   {\sum\limits_{\mathclap{e'\in E_{v_i}}}{N_t(e')}^\a}
		   \ge            \frac{{(\r_\e^{-1}\mu(e))}^\a}
		{\sum\limits_{\mathclap{e'\in E_{v_i}}}{{(\r_\e\mu(e'))}^\a}}
           = \r_\e^{-2\a}
           \frac{{\mu(e)}^\a}
                {\sum\limits_{\mathclap{e'\in E_{v_i}}}{{\mu(e')}^\a}}.
    \end{align}}
	Therefore, using that $\mu$ is an equilibrium, $\{N_t(e)\}_{t \ge T}$ is dominated from below by a Poisson process with intensity 
	$$\r_\e^{-2\a}\frac{\mu(e)^\a}
 { 
  \sum\limits_{{\mathclap{\substack{e'\in E_{v_1}}}}}
 {		{\mu(e')^\a}{}}}
	 + \r_\e^{-2\a}\frac{\mu(e)^\a}
 { 
  \sum\limits_{{\mathclap{\substack{e'\in E_{v_2}}}}}
	{		{\mu(e')^\a}{}}} = \r_\e^{-2\a} \mu(e).
	$$
	Hence, almost surely, $X_-(e) \ge \r_\e^{-2\a} \mu(e)$. Similar arguments yield that $X_+(e) \le \r_\e^{2\a} \mu(e)$. 
\end{proof}

The proof of Lemma \ref{noPercLem} is the most challenging part of
the auxiliary results. The main ingredient is
\cref{lemma:loc_secondboundonC} below. It states that $X_-(e)$ is
bounded away from 0 with a high probability, even when conditioning
on $\Fe$, where for $e \in E$ we let 
\[\Fe := \sigma(\{P_v\}_{v \not \in \cup_{e' \in E_e}e'})\]
denote the $\sigma$-algebra generated by all Poisson processes at nodes that are at distance at least~2 away from the edge $e$. 
\begin{lemma}[Compact containment of $\C(e)$] 
	\label{lemma:loc_secondboundonC}
	For every $\e > 0$ there exists $\de > 0$ such that almost surely
	$$\inf_{e \in E}\P( X_-(e) \ge \de\,|\, \Fe) \ge 1 - \e.$$
	In particular, $\P( X_-(e) > 0\,|\, \Fe) = 1$ for every $e \in E$.

\end{lemma}
Hence, we conclude from dependent percolation theory in the form of \cite[Theorem 0.0]{domProd} that the edges violating the lower bound are restricted to finite, well-separated islands. To make the presentation self-contained, we give an elementary direct proof.

\begin{proof}[Proof of Lemma \ref{noPercLem}]
    Let $v \in V$ be arbitrary. We resort to a first-moment argument and show
    that for sufficiently small $\de$ the expected number of length-$n$
    self-avoiding paths of $\de$-unstable edges tends to 0 as $n \to \infty$.
    Since the number of length-$n$ self-avoiding paths in $G$ starting from $v$
    is at most $\De^n$, it suffices to show that the probability for any fixed
    self-avoiding path to consist of $\de$-unstable edges only is of the order
    at most $(2\De)^{-n}$. 

	Note that the number of edges that are at graph distance at most 3 of a given edge is bounded above by $2\De^3$. Hence, any self-avoiding path of length $n$ contains at least $n / (2\De^3)$ edges that are of pairwise distance at least 4, so that by Lemma \ref{lemma:loc_secondboundonC}, the probability that they are all $\de$-unstable is at most $\e(\de)^{n / (2\De^3)}$. In particular, choosing $\de > 0$ such that $\e(\de)^{1 / (2\De^3)} \le 1 / (2\De)$ concludes the proof.
\end{proof}

It remains to prove Lemma \ref{lemma:loc_secondboundonC}. To that end, we invoke a conditioning argument in the spirit of \cite[Lemma 5.2]{benaim} to provide a bootstrap result propagating large edge weights at a current time point to a considerable duration into the future.
For $k, \ell \ge 1$ put 
 \begin{align}
   a_{k, \ell} := \De^{-\frac1{1 - \a}}
    2^{k(\ell - 1) - k \sum_{2 \le i \le \ell}\a^i}.
 \end{align}
\begin{lemma}[Bootstrapped lower bound]
	\label{lemma:loc_HilfsLemma}
	There exists a constant $c = c(\a, \De) > 0$ such that for all $k$ large enough, all $\ell\ge 1$ and all $e \in E$, almost surely, 
	\begin{align}\label{eqn:condProb}
  &\P(N_{2^{k(\ell + 1)}}(e)
		\le a_{k, \ell + 1} \text{ and }N_{2^{k\ell}}(e) \ge a_{k, \ell}
 \,|\, \Fe )
		\le e^{
 			-c2^{k\ell(1 - \a)}
 }.
 \end{align}
\end{lemma}
\begin{proof}
	Let $\e > 0$ be arbitrary.
	Then, the Poisson concentration inequality \cite[Lemma 1.2]{penrose} implies that the event
	 \begin{align}
		 A_e:=\{N_{2^{k(\ell + 1)}}({e'}) \le		 2^{1 + \e +  k(\ell +
         1)}\ \text{for all}\ e'\in E_{e}\}\label{eqn:upper_bound_weights}
 \end{align}
		has a probability tending to $1$ with an error decaying exponentially in
        $2^{k\ell}$.
 
	Further, under the event $\{N_{2^{k\ell}}(e)\ge a_{k, \ell}\} \cap A_e$, 
 \cref{eqn:polya_weight} has a lower bound for times
 $t \in T_\ell : = [2^{k\ell}, 2^{k(\ell + 1)}]$ given by
 \begin{align}
 {\ms{pol}_{v, e}(N_t)}
  = \frac
  {N_t(e)^\a}{\sum_{e'\in E_v}N_t({e'})^\a}
 \ge
 \frac
  {a_{k, \ell}^\a}{a_{k, \ell}^\a
   + \De
  2^{(1 + \e + k(\ell + 1))\a}}.
 \end{align}
	Since $\De a_{k, \ell + 1} = a_{k, \ell}^\a 2^{k (\ell + \a) (1 - \a)}$, putting $\a_\e = \a (1 + \e)$, we deduce that
	\begin{align}
		{\ms{pol}_{v, e}(N_t)}&\ge 	 
		\frac1{1 + 2^{\a_\e} 2^{k (\ell + 1) \a}
    a_{k, \ell + 1}^{-1} 2^{k (\ell + \a)(1 - \a)}
  }
		 = \frac1{1 + 2^{\a_\e + k(\ell + 1 - (1 - \a)^2)}
    a_{k, \ell + 1}^{-1}
}.
	\end{align}
	Hence,
	\begin{align}
	{\ms{pol}_{v, e}(N_t)}
		\ge \frac {2^{-\a_\e + k((1 - \a)^2 - \ell - 1)} a_{k, \ell + 1}}
    {1 + 2^{-k\a}}
 = :{G_{k, \ell}}\;.
 \end{align}
 In particular, for $t\in T_\ell$, 
 the Poisson processes $P_v$ and $P_w$ having points in
 $[2^{k\ell}, t]\times U_e$
 where $U_e\subseteq[0, 1]$ and $|U_e|\ge
 G_{k, \ell}$ implies weight increases of the edge $e$
 (by one or more) in the time interval $[2^{k\ell}, t]$.
 Therefore, using \cref{eqn:upper_bound_weights}, we find a constant $c$ such  that
 \begin{align}
     &\P\left(N_{2^{k(\ell + 1)}}(e)
             \le a_{k, \ell + 1}\text{ and } 
 N_{2^{k\ell}}(e)\ge a_{k, \ell}\right)\\
        &\quad\le\P\left(N_{2^{k(\ell + 1)}}(e)-N_{2^{k\ell}}(e)\le 
 a_{k, \ell + 1}\text{ and }
 N_{2^{k\ell}}(e)\ge a_{k, \ell}\right)\\
        &\quad\le 
  \P\left(P_v(T_\ell\times U_e)
  + P_w(T_\ell\times U_e)\le a_{k, \ell + 1}
  \text{ and } N_{2^{k\ell}}(e)\ge a_{k, \ell}\right)
  + {O}\big(e^{-c 2^{k\ell}}\big)\\
 &\quad \le \P
 \left(
  \ms{Poi}\left(
  2|T_\ell| G_{k, \ell}\right)
  \le a_{k, \ell + 1}\label{eqn:poisson_param}
  \right) + {O} \big( e^{-c 2^{k\ell}} \big).
 \end{align}
 Since
 \begin{align}
 2|T_\ell|G_{k, \ell} = 
 a_{k, \ell + 1}2^{1 - \a_\e + k(1 - \a)^2}
 \frac{1 - 2^{-k}}{1 + 2^{-k\a}}
 \end{align}
 there exists $\tilde c > 1$ such that $
 2|T_\ell|G_{k, \ell}
  > \tilde c\cdot a_{k, \ell + 1}
  $ holds for all $k$ large enough.
 Thus, by \cite[Lemma 1.2]{penrose},
 \begin{align}
 \P\left(N_{2^{k(\ell + 1)}}(e)\le a_{k, \ell + 1}\text{ and } 
 N_{2^{k\ell}}(e)\ge a_{k, \ell}\right)&\le
  \exp\left((1 - \tilde c + \log(\tilde c))a_{k, \ell + 1}\right)\;.
 \end{align}
 Now, we conclude the proof by noting that $a_{k, \ell + 1} \ge \De^{-\frac1{1 - \a}}
		2^{k\ell(1 - \a)}$ and that  the coefficient of
 $a_{k, \ell + 1}$ is negative. 
 \end{proof}
%
%
\cref{lemma:loc_HilfsLemma} propagates the lower bounds
on the edge weights through time. As $t \ra \infty$,
these lower bounds allow to exclude vanishing edge weights.

\begin{proof}[Proof of Lemma \ref{lemma:loc_secondboundonC}] Defining 
	$F = \cap_{\ell \ge 1}\{N_{2^{k\ell}}(e)\ge a_{k, \ell}\}$ and 
$\de_0 = \liminf_{\ell \to \infty}a_{k, \ell}2^{-k\ell} > 0,$
	we note that 
	$X_-(e) \ge \de_0$ holds under the event $F$. 	Hence, it remains to establish a lower bound on $\P(F|\Fe)$. 

Since $a_{k, 1} \le 1$, the bound $N_{2^k}(e) \ge a_{k, 1}$ holds for any $k \ge 1$. Hence, if the event $F$ does not occur, then there exists $\ell_0 \ge 1$ such that $N_{2^{k(\ell_0 + 1)}}(e)		\le a_{k, \ell_0 + 1}$ and $N_{2^{k\ell_0}}\ge a_{k, \ell_0}$. In particular, by Lemma \ref{lemma:loc_HilfsLemma}, almost surely,
	$$1 - \P(F\,|\, \Fe) \le \sum_{\ell \ge 1}\P(N_{2^{k(\ell + 1)}}(e)
		\le a_{k,\ell + 1} \text{ and }N_{2^{k\ell}}\ge a_{k, \ell}\,|\,\Fe) \le \sum_{\ell \ge 1}e^{		-c2^{k\ell (1 - \a)} },$$
which becomes smaller than $\e$ for sufficiently large $k > 0$.
\end{proof}

\section{Regular graphs}\label{seg:regular}
%
%
Throughout this section, we assume $G$ to be $\De$-regular.  The key to the
proof of parts (2) and (3) of \cref{mainThm} is the following bootstrap property.

\begin{lemma}[Bootstrap on regular graphs]\label{lemma:multi_tightenBounds}
    Let $a < 2/\De$, $b > 2/\De$ and $e \in E$ be such that $\C(e) \subseteq
    [a,b]$ holds almost surely for all $e' \in E_e$. Furthermore, define the
    function $$f(r, s) :=\frac{2r^\a}{r^\a + (\De - 1)s^\a}.$$ Then,
    $\C(e)\subseteq[f(a, b), f(b, a)]$ holds almost surely. In particular, for
    $G = \Z$ there exist $a'\in [a, 1)$ and $b' \in (1,b]$ such
    that $a'/b' \ge {(a / b)}^\alpha$ and $\C(e)\subseteq[a', b']$ holds almost
    surely.  \end{lemma} Before proving Lemma \ref{lemma:multi_tightenBounds}
    we explain how it implies part (3) of \cref{mainThm}.

%
%
\begin{proof}[Proof of part (3) of \cref{mainThm}]
	%
	%
First, $X_-(e)$ is strictly positive by \cref{noPercLem} using a similar
argument to the proof of part (1) of \cref{mainThm}. Hence,
\cref{lemma:multi_tightenBounds} gives that $X_-(e) \ge X_-(e)^\a 2^{-\a}$,
	i.e., $X_-(e) \ge 2^{-\a / ( 1 - \a)}$.
	%
	%
	In other words, almost surely $\C(e) \subseteq [a_1, b_1]$ where $a_1 := 2^{-\a / ( 1 - \a)}$ and $b_1 := 2$. Applying Lemma \ref{lemma:multi_tightenBounds} iteratively yields sequences ${\{a_i\}}_{i\ge 1}$ and
        ${\{b_i\}}_{i\ge 1}$ such that
        \begin{enumerate}
                \item ${\{a_i\}}_{i\ge 1}$ is increasing and bounded above
                        by $1$,
                \item ${\{b_i\}}_{i\ge 1}$ is decreasing and bounded below
                        by $1$,
                \item $b_{i + 1}/a_{i + 1}\le (b_i/a_i)^\a < b_i/a_i$, and
		\item $\C(e)\subseteq \bigcap_{i \ge 1}[a_i, b_i]$
                       holds almost surely for all $ e\in E$.
        \end{enumerate}
	Since the first three items imply that $a_i$ and $b_i$ converge to 1, we arrive at 
	\begin{align}
\P(\C(e) = 1) = \P(\cap_{i\ge 0}\{
                 \C(e)\subseteq[a_i,b_i]\})
                 = \lim_{i \to \infty}\P(
                \C(e)\subseteq[a_i,b_i])
                =  1,\label{eqn:regular_almost_sure_limit}
        \end{align}
	as asserted.
\end{proof}

%
%
Next, we prove Lemma~\ref{lemma:multi_tightenBounds}.
\begin{proof}[Proof of Lemma \ref{lemma:multi_tightenBounds}]
        Let
        \begin{align}
                G_{a,b} := \cap_{e'\in E_e}\{\C(e')\subseteq[a,b]\}
        \end{align}
        denote the event that $\C(e')\subseteq [a, b]$ holds for all
        $e'\in E_e$.
        Under this event, $X_t^{\{v,w\}}$ gains mass
        at a rate of at least
        \begin{align}
		a'' = f(at, bt) = f(a, b),
	\end{align}
        for $t$ large enough. 
        More precisely, under $G_{a,b}$ one can find a sequence
        ${\{\e_t\}}_{t\ge 0}$ with $\e_t\searrow 0$ such
        that almost surely
        \begin{align}
                N_t^{\{v,w\}}\ge P_v([0,t]\times U_{v,t})
                + P_w([0,t]\times U_{w,t}) + 1
        \end{align}
	for all $t>0$ where
        $|U_{v,t}| = |U_{w,t}| = a''/2- \e_t$.
        Analogous arguments to the one in the proof of
        \cref{lemma:loc_firstboundonC} give that almost surely
        \begin{align}
                \liminf_{t\to\infty} X_t^{\{v,w\}} \ge a',
        \end{align}
	where  {$a':=a\vee a''$}.
        Similar arguments give the upper bound $b' := b\wedge b''$ where $b''
        := f(b,a)$. { In the special case $\Delta=2$ (i.e.\ $G=\Z$)
        we find $a'/b'\ge {(a/b)}^{\a}$ as desired.}
        \end{proof}
%
%
Now, let $G$ be $\De$-regular for some $\De \ge 2$. In this case we do
not have $a'/b'\ge {(a/b)}^{\a}$ to prove \cref{mainThm}, but we still have the
sequences ${\{a_i\}}_{i\ge 1}$ and ${\{b_i\}}_{i\ge 1}$ with
$b_{i+1}/a_{i+1}\le b_i/a_i$. The idea is to show that for $\a$ sufficiently close to
$1/2$ the inequality is strict.
\begin{lemma}\label{lemma:improve_at_least_one}
        For $\a=1/2+\varepsilon$ with $\varepsilon$ small enough and $0<a<2/\Delta<b<2$
        we have $f(a,b)>a$ or $f(b,a)<b$.
\end{lemma}
\begin{proof}
        Assume both statements are wrong, i.e. $f(a,b)\le a$ and $f(b,a)\ge b$, then
        this gives 
        \begin{align}
                \rho \le\rho^{2\a}
                \frac{1+\frac{\rho^{-\a}}{\Delta-1}}{1+\frac{\rho^{\a}}{\Delta-1}}
                \label{eqn:regular_to_contradict}
        \end{align}
for $\rho:=b/a$.
        Insert $\alpha=1/2+\varepsilon$ and take log on both sides to get
        \begin{align}
                0\le 2\varepsilon\log\rho +
                \log\Big(1+\frac{\rho^{-1/2-\varepsilon}}{\Delta-1}\Big)
                - \log\Big(1+\frac{\rho^{1/2+\varepsilon}}{\Delta-1}\Big)\;.
        \end{align}
        Note that equality holds for $\rho=1$ so if the right-hand side is decreasing in $\rho$ then
        this contradicts the above inequality for $\varrho>1$. To that end, note that
        the derivative
        \begin{align}
                \frac1\rho\Big(2\varepsilon
                -(1/2+\varepsilon)
                \Big(\frac{1}{1+(\De-1)\rho^{-1/2-\varepsilon}}
                +
                \frac1{1+(\De-1)\rho^{1/2+\varepsilon}}
        \Big)\Big)
        \end{align}
        of the right-hand side is bounded above for all $\rho\ge 1$ by
        \begin{align}
                \frac1\rho\Big(2\varepsilon
                -\frac{(1/2+\varepsilon)}{\Delta}\Big)
        \end{align}
        which is negative for $\varepsilon$ small enough and hence we get a
        contradiction to \ref{eqn:regular_to_contradict}.
\end{proof}
Not improving both bounds at each application of the bootstrap property means that the proof for a uniform lower bound on
$\mc C(e)$ in part (1) of \cref{mainThm} does not work since
we do not immediately get a contradiction. The following Lemma helps remedy
this.
\begin{lemma}\label{lemma:lower_threshold_for_improvement}
        For $a<2\Delta^{-1/(1-\a)}$ we have $f(a,2)>a$.
\end{lemma}
\begin{proof}
        Straightforward calculation as
        \begin{align}
                f(a,2)=\frac{2a^{\a}}{a^{\a}+(\Delta-1)2^{\a}}
                > \frac{2a^{\a}}{2^{\a}\Delta}
                =a \frac{2^{1-\a}}{a^{1-\a}\Delta}
		>a.&\qedhere
        \end{align}
\end{proof}
\begin{proof}[Proof of part (2) of \cref{mainThm}]
        Take the $\varepsilon$ from \cref{lemma:improve_at_least_one} and
        consider $\alpha=1/2+\varepsilon$. $X_-(e)$ is bounded above by
        \cref{lemma:loc_firstboundonC} and strictly positive by
        \cref{noPercLem} and \cref{lemma:lower_threshold_for_improvement} using
        analogous arguments as in the proof of part (1) of this theorem.  So we
        find $a \le 2/\Delta$ and $b \ge 2/\Delta$ such that
        $\mathcal{C}(e)\subseteq[a,b]$ for all $e\in E$. 
	    We choose $a$ maximal and $b$ minimal with that property. To derive a contradiction, assume that $a < b$.
        Then, by \cref{lemma:multi_tightenBounds} and \cref{lemma:improve_at_least_one} we can tighten the bounds as
        \begin{align}
                \mathcal{C}(e)\subseteq [a\vee f(a,b), b\wedge
		f(b,a)] \subsetneq [a, b]\;\label{eqn:new_bounds_on_C}.
        \end{align}
        This contradicts the maximality/minimality of $a$ and $b$,
        and since $f(a,b)\le 2/\De \le f(b,a)$ therefore gives $a = b = 2/\Delta$.
\end{proof}

\subsection*{Acknowledgement.}The authors thank both anonymous referees for the
careful reading of the manuscript and the constructive feedback. Their comments
and suggestions substantially helped to improve the presentation of the
material. The authors thank M.~Holmes and V.~Kleptsyn for illuminating
discussions and ideas for future work.
\bibliography{./main.bbl}
\bibliographystyle{abbrv}

\end{document}